\documentclass[12pt]{article}
\usepackage{amsmath}
\usepackage{amssymb}
\usepackage{graphicx}
\usepackage{hyperref}

\makeatletter

\usepackage{amsthm}\usepackage{enumerate}\usepackage{fullpage}

\theoremstyle{plain}
\newtheorem{lemma}{Lemma}
\newtheorem{theorem}[lemma]{Theorem}

\newtheorem{claim}{Claim}

\theoremstyle{definition}

\newtheorem{question}[lemma]{Question}

\makeatother

\begin{document}

\title{On choosability with separation of planar graphs with lists of different
sizes}

\author{H. A. Kierstead %
\thanks{School of Mathematical and Statistical Sciences, Arizona State University,
Tempe, AZ 85287, USA. E-mail: \texttt{hal.kierstead@me.com}. %
Research of this author is supported in part by NSA grant H98230-12-1-0212.} \and Bernard Lidick\'y %
\thanks{Department of Mathematics, University of Illinois. E-mail: \texttt{lidicky@illinois.edu}%
} }

\date{\today}

\maketitle
 
\begin{abstract}
A \emph{$(k,d)$-list assignment} $L$ of a graph $G$ is a mapping
that assigns to each vertex $v$ a list $L(v)$ of at least $k$ colors
and for any adjacent pair $xy$, the lists $L(x)$ and $L(y)$ share
at most $d$ colors. A graph $G$ is $(k,d)$-choosable if there exists
an $L$-coloring of $G$ for every $(k,d)$-list assignment $L$.
This concept is also known as choosability with separation.

It is known that planar graphs are $(4,1)$-choosable but it is not
known if planar graphs are $(3,1)$-choosable. We strengthen the result
that planar graphs are $(4,1)$-choosable by allowing an independent
set of vertices to have lists of size 3 instead of 4. 
\end{abstract}

\section{Introduction}

Given a graph $G$, a \textit{list assignment} $L$ is a mapping assigning
to each vertex $v\in V(G)$ a list of colors $L(v)$. An \emph{$L$-coloring}
is a vertex coloring $\varphi$ such that $\varphi(v)\in L(v)$ for
each vertex $v$ and $\varphi(x)\neq\varphi(y)$ for each edge $xy$.
A graph $G$ is said to be \emph{$k$-choosable} if there is an $L$-coloring
for each list assignment $L$ where $|L(v)|\geq k$ for each vertex
$v$. The minimum such $k$ is known as the \emph{choosability} of
$G$, denoted $\chi_{\ell}(G)$. A graph $G$ is said to be \emph{$(k,d)$-choosable}
if there is an $L$-coloring for each list assignment $L$ where $|L(v)|\geq k$
for each vertex $v$ and $|L(x)\cap L(y)|\leq d$ for each edge $xy$.

This concept is called \emph{choosability with separation}, since
the second parameter may force the lists of adjacent vertices to be
somewhat separated. If $G$ is $(k,d)$-choosable, then $G$ is also
$(k',d')$-choosable for all $k'\geq k$ and $d'\leq d$. A graph
is $(k,k)$-choosable if and only if it is $k$-choosable. Clearly,
all graphs are $(k,0)$-choosable for $k\geq1$. Thus, for a graph
$G$ and each $1\leq k<\chi_{\ell}(G)$, there is some threshold $d\in\{0,\dots,k-1\}$
such that $G$ is $(k,d)$-choosable but not $(k,d+1)$-choosable.

The concept of choosability with separation was introduced by Kratochv\'{\i}l,
Tuza, and Voigt~\cite{ktv}. They used the following, more general
definition. A graph $G$ is \emph{$(p,q,r)$-choosable}, if for every
list assignment $L$ with $|L(v)|\geq p$ for each $v\in V(G)$ and
$|L(u)\cap L(v)|\leq p-r$ whenever $u,v$ are adjacent vertices,
$G$ is $q$-tuple $L$-colorable. Since we consider only $q=1$ in
this paper, we use a simpler notation. They investigate this concept
for both complete graphs and sparse graphs. The study of dense graphs
were extended to complete bipartite graphs and multipartite graphs
by F\"uredi, Kostochka, and Kumbhat~\cite{zkk1,zkkarxiv}.

Thomassen~\cite{thomassen1994} proved that planar graphs are 5-choosable,
and hence they are $(5,d)$-choosable for all $d$. Voigt~\cite{95Voigt}
constructed a non-$4$-choosable planar graph, and there are also
examples of non-$(4,3)$-choosable planar graphs. Kratochv\'{\i}l, Tuza,
and Voigt~\cite{ktv} showed that all planar graphs are $(4,1)$-choosable
and asked:
\begin{question}[\cite{ktv}]
Are all planar graphs $(4,2)$-choosable?
\end{question}
Voigt~\cite{93Voigt} also constructed a non-$3$-choosable triangle-free
planar graph. \v{S}krekovski~\cite{riste} observed that there are examples
of triangle-free planar graphs that are not $(3,2)$-choosable, and
posed:
\begin{question}[\cite{riste}]\label{q:31}
Are all planar graphs $(3,1)$-choosable?
\end{question}
Kratochv\'{\i}l, Tuza and Voigt~\cite{ktv} proved a partial case of 
Question~\ref{q:31} by showing that every triangle-free planar graph is $(3,1)$-choosable.

Choi et. al~\cite{CLS} proved that every planar graph without $4$-cycles
is $(3,1)$-choosable and that every planar graph without $5$-cycles
and $6$-cycles is $(3,1)$-choosable.

In this paper we give a strengthening of the result that every planar
graph is $(4,1)$-choosable by allowing some vertices to have lists
of size three. In a $(4,1)$-list assignment $L$ on $G$, for every
$uv\in E(G)$ holds that $|L(u)\cup L(v)|\geq7$. In a $(3,1)$-list
assignment $L$, for every $uv\in E(G)$ holds that $|L(u)\cup L(v)|\geq5$.
An intermediate step is to investigate the case where for every $uv\in E(G)$
holds that $|L(u)\cup L(v)|\geq6$.

A \emph{$(*,1)$-list assignment} is a list assignment $L$ where
$|L(v)|\geq1$ and $|L(u)\cap L(v)|\leq1$ for every pair of adjacent
vertices $u,v$.

The main result of this paper is the following theorem. 

\begin{theorem}\label{thm:6} Let $G$ be a planar graph and $I\subseteq V(G)$
be an independent set. If $L$ is a $(*,1)$-list assignment such
that $|L(v)|\geq3$ for every $v\in I$ and $|L(v)|\geq4$ for every
$v\in V(G)\setminus I$ then $G$ has an $L$-coloring. \end{theorem}

The following theorem shows it is not possible to strengthen 
Theorem~\ref{thm:6} by allowing $|L(v)|\geq2$ for every vertex
$v\in V(G)$ and requiring that $|L(u)\cup L(v)|\geq6$ for every
$uv\in E(G)$.

\begin{theorem}\label{thm:3necessary} For every $k$ there exists
a planar graph $G$ and a $(*,1)$-list assignment $L$ such that
$|L(v)|\geq2$ for every $v\in V(G)$, $|L(u)\cup L(v)|\geq k$ for
every $uv\in E(G)$, and $G$ is not $L$-colorable. \end{theorem}

We first give some notation. In the next section, we prove Theorem~\ref{thm:6}
using Thomassen's precoloring extension method. In the last section
we show a construction proving Theorem~\ref{thm:3necessary}.

\subsection{Notation}

\global\long\def\cin{\mathop{\mathrm{Int}}}
 \global\long\def\cex{\mathop{\mathrm{Ext}}}
 Given a graph $G$ and a cycle $K\subset G$, an edge $uv$ of $G$
is a \emph{chord} of $K$ if $u,v\in V(K)$, but $uv$ is not an edge
of $K$. If $G$ is a plane graph, then let $\cin_{K}(G)$ be the
subgraph of $G$ consisting of the vertices and edges drawn inside
the closed disc bounded by $K$, and let $\cex_{K}(G)$ be the subgraph
of $G$ obtained by removing all vertices and edges drawn inside the
open disc bounded by $K$. In particular, $K=\cin_{K}(G)\cap\cex_{K}(G)$.
Finally, denote the characteristic function of a set $S$ by $\iota_{S}$.
So $\iota_{S}(x)=1$ if $x\in S$; else $\iota_{S}(x)=0$.


\section{Main theorem}

In this section, we prove Theorem~\ref{thm:6} by proving a slightly
stronger theorem that is more amenable to induction. Observe that
any list assignment satisfying the assumptions of Theorem~\ref{thm:6}
also satisfies the conditions of the following theorem.

\begin{theorem}\label{thm:6thomassen} Let $G$ be a plane graph
with outer face $F$ and let $P$ be a subpath of $F$ containing
at most two vertices. Let $I\subseteq V(G-P)$ be an independent set.
If $L$ is a $(*,1)$-list assignment satisfying the following conditions: 
\begin{enumerate}[(i)]\itemsep0em 
\item $|L(v)|\geq4-\iota_{I}(v)-\iota_{V(F)}(v)-2\iota_{V(P)}(v)$ for $v\in V(G)$,
\label{cond:in4} 
\item $P$ is $L$-colorable,\label{cond:LP} 
\item for every $v\in I$ there is at most one $p\in N(v)\cap V(P)$ with
$(L(p) \cap L(v)) \neq \emptyset$,\label{cond:IP} 
\end{enumerate}
then $G$ is $L$-colorable. \end{theorem}

\begin{proof} Let $G=(V,E)$ and $L$ be a counterexample where $|V|+|E|$
is as small as possible. Moreover, assume that the sum of the sizes
of the lists is also as small as possible subject to the previous
condition. Define $L(uv)=L(u)\cap L(v)$ if $uv\in E$; else $L(uv)=\emptyset$.
Since $G$ is minimal, we have:

\begin{claim}\label{cl0}
For all edges $uv,vw,uw\in E\setminus E(P)$
\begin{enumerate}[(1)]\itemsep0em 
\item \label{L=00003D1}$|L(uv)|=1$;
\item \label{Lu}$L(u)=\bigcup_{v\in N(u)}L(uv)$; and
\item \label{3. Ltri}$L(uv)=L(vw)$ implies $L(uv)=L(uw)$ for every triangle
$uvwu$.
\end{enumerate}
\end{claim}

\begin{proof}
For (\ref{L=00003D1}), note that $|L(uv)|\leq1$, and if $L(uv)=\emptyset$
then it suffices to  $L$-color $G-uv$, which is possible by
minimality. For (\ref{Lu}), the definitions imply $L(u)\supseteq\bigcup_{v\in N(u)}L(uv)$,
and if $\gamma\in L(u)\setminus\bigcup_{v\in N(u)}L(uv)$ then $L$-coloring
$G-u$, and then coloring $u$ with $\gamma$ yields an $L$-coloring
of $G$. Finally consider (\ref{3. Ltri}). By (\ref{L=00003D1}),
there exists a color $\gamma$ with $L(uv)=\{\gamma\}=L(vw)$. Thus
$\gamma\in L(u)\cap L(w)$, so by definition and (\ref{L=00003D1}),
$L(uw)=\{\gamma\}$.
\end{proof}

\begin{claim}\label{2conn}$G$ is 2-connected. In particular, $F$
is a cycle.\end{claim}
\begin{proof}
Suppose not. Then there exists $v\in V$ and two induced connected
subgraphs $G_{1}$ and $G_{2}$ of $G$ where $G_{1}\cap G_{2}=v$
and $G_{1}\cup G_{2}=G$. Moreover, both $G_{1}$ and $G_{2}$ have
at least two vertices. By symmetry assume that $P\subseteq G_{1}$.
By the minimality of $G$, there exists an $L$-coloring $\varphi$
of $G_{1}$. Let $L_{2}$ be a list assignment on $V(G_{2})$ such
that $L_{2}(u)=\{\varphi(v)\}$ if $u=v$, and $L_{2}(u)=L(u)$ otherwise.
Since $L_{2}$ and $G_{2}$ satisfy the assumptions of Theorem~\ref{thm:6thomassen},
there exists an $L_{2}$-coloring $\psi$ of $G_{2}$. Colorings $\varphi$
and $\psi$ coincide on $v$; hence $\varphi\cup\psi$ is an $L$-coloring
of $G$, a contradiction.
\end{proof}

\begin{claim}\label{cl:FnotI}
(1) $|N(v) \cap V(P)| \leq 1$ for all $v\in I$, and (2) $V(F)\setminus(I\cup V(P)) \neq \emptyset$.
\end{claim}

\begin{proof}
The minimality of $G$ and (\ref{cond:IP}) imply (1).
Using Claim~\ref{2conn}, $F-P$ is a path. Since  $I$  is independent, if $V(F)\subseteq I\cup V(P)$ then $|I\cap V(F)|=1$, contradicting (1).
\end{proof}

\begin{claim}\label{cl:septriangle} $G$ does not contain a separating
triangle with a vertex in $I$. \end{claim} \begin{proof} Let $T=xyz$
be a separating triangle in $G$ and let $x\in I$. Assume that $P\subseteq\cex_{T}(G)$
and $|V(\cin_{T}(G))|\geq4$. By the minimality of $G$, there exists
an $L$-coloring $\varphi$ of $\cex_{T}(G)$.

Let $G':=\cin_{T}(G)-z$, $I':=I\setminus V(\cex_{T}(G))$ and $P'=xy$.
Define a list assignment $L'$ on vertices $u\in V(G')$ in the following
way: 
\[
L'(u)=\begin{cases}
\varphi(u) & \text{ if }u\in\{x,y\},\\
L(u)-\varphi(z) & \text{ if }uz\in E(G-P'),\\
L(u) & \text{otherwise}.
\end{cases}
\]
Since $x\in I$, no neighbor of $x$ is in $I'$. Thus condition (\ref{cond:IP})
of Theorem~\ref{thm:6thomassen} is satisfied for $G',I',P'$ and
$L'$. Condition (\ref{cond:LP}) is witnessed by $\varphi$. Since each vertex
$u\in N_{G'}(z)$ is on the outer face of $G'$, but not $G$, it
is straightforward to check that (\ref{cond:in4}) is satisfied. Hence
$G'$ has an $L'$-coloring $\varphi$. The coloring $\varphi\cup\psi$
is an $L$-coloring of $G$, a contradiction. \end{proof}

\begin{claim}\label{cl:chord} If $xy$ is a chord of $F$ then neither
$x$ nor $y$ is in $V(P)$, and there exists $z\in I\cap V(F)$ such
that  $|L(z)|=2=d(z)$, $L(zx)\neq L(zy)$, and $xzy\subseteq F$.
\end{claim}

\begin{proof} Suppose $xy\in E$ is a chord of $F$. Let $G_{1}$
and $G_{2}$ be subgraphs of $G$ where $G_{1}\cap G_{2}=xy$ and
$G_{1}\cup G_{2}=G$. Since $xy$ is a chord, both $G_{1}$ and $G_{2}$
have at least three vertices. By symmetry assume that $P\subset G_{1}$.

First suppose $G_{2}$ contains exactly three vertices, say $x,y,z$.
Using Claim~\ref{cl0},  $2\leq|L(z)|=|L(zx)\cup L(zy)|\leq d(z)\leq2$.
So $|L(z)|=2=d(z)$ and $L(zx)\neq L(zy)$. By condition~(\ref{cond:in4}),
$|L(z)|=2$ implies $z\in I\cap V(F)$. Thus $xzy\subseteq F$, since
$x$ and $y$ are the only possible neighbors of $z$. Finally, since
$L(zx)\ne L(zy)$, Claim~\ref{cl0}.\ref{3. Ltri} implies $L(xy)\nsubseteq L(zx)\cup L(zy)$.
Thus $|L(x)|,|L(y)|\geq2$, and so $x,y\notin P$.

Now suppose for a contradiction that $G_{2}$ has at least four vertices.
Define $G_{1}'$ in the following way. If there exists a vertex $v\in V(G_{2})\cap I$
such that $v$ is adjacent to both $x$ and $y$ then $G_{1}'$ is
obtained from $G_{1}$ by adding a new vertex $v'$ adjacent to $x$
and $y$ to the outer face of $G$. Moreover, let $I'=(I\cap G_{1})\cup\{v'\}$
and let $L'$ be an extension of $L$ by defining $L'(v')=L(vx)\cup L(vy)$.
Notice that $v$ is unique if it exists, since Claim~\ref{cl:septriangle}
implies $G$ has no separating triangles that contain a vertex of
$I$. If no such $v$ exists, let $G_{1}'=G_{1}$, $L'=L$, and $I'=I\cap V(G_{1})$.
If $G_{1}'$ contains $v'$, neither $x$ nor $y$ is in $I$. Hence
$I'$ is indeed an independent set. Using that $v'\in I'$ is on the
outer face$,$ $L'$ satisfies conditions (\ref{cond:in4},\ref{cond:LP},\ref{cond:IP}).
By the minimality of $G$, there exists an $L'$-coloring $\varphi$
of $G'$ which gives an $L$-coloring of $G_{1}$.

Define a list assignment $L_{2}$ on $V(G_{2})$ by $L_2(u)=\{\varphi(u)\}$
if $u\in\{x,y\}$, else $L_{2}(u)=L(u)$. We wish to use $xy$
as $P$. Conditions (\ref{cond:in4},\ref{cond:LP}) of Theorem~\ref{thm:6thomassen} hold since $G$ satisfies 
them. For (\ref{cond:IP}), consider a vertex $w\in I$ with $\{x,y\}\subseteq N(w)$.
As remarked above, $w=v$. Since $L'(v)=L(vx)\cup L(vy)$, and $\varphi$
is an $L'$-coloring of $G'$, $ $$ $there exists $u\in\{x,y\}$
with $\varphi(v)\in L(vu)$. Then $\varphi(v)\ne\varphi(u)$ implies
$\varphi(u)\notin L(v)$, and (\ref{cond:IP}) holds. By the minimality
of $G$, there exists an $L_{2}$-coloring $\psi$ of $G_{2}$. Colorings
$\varphi$ restricted to $G_{1}$ and $\psi$ coincide on $xy$; hence
$\phi\cup\psi$ is an $L$-coloring of $G$, a contradiction.
\end{proof}

By the minimality of the sum of the sizes of the lists, we can assume
$|V(P)|\geq1$. Let $F=v_{0}v_{1}v_{2}v_{3}\dots v_{t}$, where $v_{0}\in V(P)\subseteq\{v_{0},v_{1}\}$,
identifying index $i$ with index $i+t+1$. Choose $v_{i}\in V(F)\setminus(I\cup V(P))$ with minimum index
$i$. Such an index exists by Claim~\ref{cl:FnotI}. Claim~\ref{cl:chord} implies $v_{i}v_{i-2}$ is not 
a chord, and condition (\ref{cond:in4}) implies $L(v_{i})-L(v_{i}v_{i-1})-L(v_{i}v_{i+1})\ne\emptyset$.

Select a set $X\subseteq\{v_{i},v_{i+1},v_{i+2}\}$ and an $L$-coloring
$\varphi$ of $X$ by the following rules: 
\begin{enumerate}[(X1)]\itemsep0em 
\item If $v_{i}v_{i+2}$ is not a chord then set $X=\{v_{i}\}$ and pick
$\varphi(v_{i})\in L(v_{i})\setminus(L(v_{i-1})\cup L(v_{i+1}))$. 
\item Else, if there is $c\in L(v_{i})\setminus(L(v_{i-1})\cup L(v_{i+1})\cup L(v_{i+2}))$,
then set $X=\{v_{i}\}$ and $\varphi(v_{i})=c$. 
\item Else set $X=\{v_{i},v_{i+1},v_{i+2}\}$. Pick:

\begin{enumerate}\itemsep0em 
\item $\varphi(v_{i+2})\in L(v_{i+2})\setminus(L(v_{i+3})\cup L(v_{i+3}v_{i+4}))$; 
\item $\varphi(v_{i})\in L(v_{i}v_{i+2})$, if $\varphi(v_{i+2})\notin L(v_{i}v_{i+2})$;
else $\varphi(v_{i})\in L(v_{i}v_{i+1})$; 
\item $\varphi(v_{i+1})\in L(v_{i+1})-\varphi(v_{i})-\varphi(v_{i+2})$. 
\end{enumerate}
\end{enumerate}

See Figure~\ref{fig-defineX} for an illustration of these rules.
Observe that exactly one of (X1), (X2), or (X3) applies and $X$ is
well defined. Also, in cases (X2) and (X3), Claim~\ref{cl:chord}
implies $v_{i+1}\in I$, and so $v_{i+2}\notin I$. Thus the sizes
of their lists are as claimed in Figure~\ref{fig-defineX}. In (X3)
either $\varphi(v_{i})\in L(v_{i}v_{i+2})$ or $\varphi(v_{i+2})\in L(v_{i}v_{i+2})$.
Also, by Claims~\ref{cl0}(\ref{3. Ltri}) and \ref{cl:chord}, $L(v_{i}v_{i+2})\not\subseteq L(v_{i+1})$. 
Hence $\varphi$ is also well defined. Moreover, $d(v_{i+1})=2$,
and so $N(v_{i+1})=\{v_{i},v_{i+1}\}$. 

Let $G'=G-X$, $I'=I\setminus X$, and $L'$ be the list assignment
on $V(G')$ defined by 
\[
L'(v)=L(v)\setminus\{\varphi(x):x\in N(v)\cap X\}.
\]
It suffices to show that $G'$, $L'$, $I'$ and $P$ satisfy the
assumptions of Theorem~\ref{thm:6thomassen}. Then by the minimality
of $G$, there is an $L'$-coloring $\psi$ of $G'$, and by the choice
of $L'$, the function $\psi\cup\varphi$ is an $L$-coloring of $G$,
a contradiction.

Now we verify that $G'$, $L'$, $I'$ and $P$ satisfy the assumptions
of Theorem~\ref{thm:6thomassen}. Since $I$ is an independent set,
so is $I'$. Let $M=\{v\in V(G'):L'(v)\neq L(v)\}$. Clearly condition~(\ref{cond:in4})
holds for vertices in $V\setminus M$. By Claim~\ref{cl:chord},
all chords have the form $v_{j}v_{j+2}$. Thus $\varphi$ was chosen
so that $M\cap V(F)=\emptyset$. Hence the condition (\ref{cond:in4})
is satisfied for $v\in V(F)$. Condition 
(\ref{cond:LP}) holds since $P$ did not change. Since $I'\subseteq I$,
Claim~\ref{cl:FnotI}(1) implies condition  (\ref{cond:IP}). 

It remains to show that every $v\in M$ satisfies condition (\ref{cond:in4}).
Let $F'$ be the outer face of $G'$. Since each vertex of $M$ has
a neighbor in $X\subseteq F$, $M\subseteq F'\setminus F$. Thus it
suffices to show that $|L'(v)|\geq|L(v)|-1$. If $|N(v)\cap X|=1$
then $|L'(v)|\geq|L(v)|-1$. Otherwise $|N(v)\cap X|\geq2$. Then
$v$ is handled by rule (X3). So $N(v)\cap X=\{v_{i},v_{i+2}\}$ and
$L(v_{i}v_{i+2})\subset C:=\{\varphi(v_{i}),\varphi(v_{i+2})\}$. If $L(vv_{i})\ne L(vv_{i+2})$
then Claim~\ref{cl0}(\ref{3. Ltri}) implies $L(v_{i}v_{i+2})\not\subset L(v)$.
Anyway, $|L(v)\cap C|\leq1$, and we are done. \end{proof}

\begin{figure}[htp]
\begin{centering}
\includegraphics{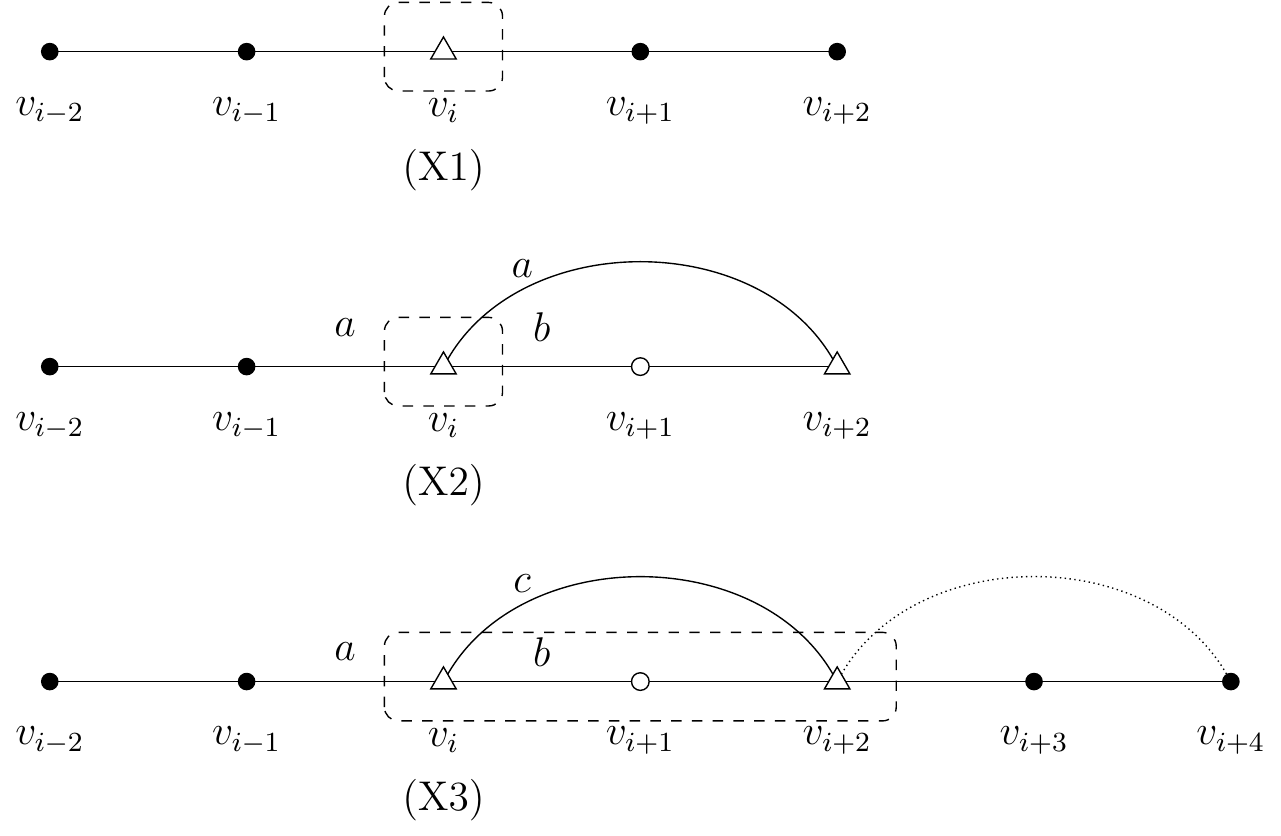} 
\par\end{centering}

\caption{Rules (X1), (X2), and (X3). A black circle is a vertex with arbitrary
list size, a white circle is a vertex with list size two, and a triangle
is a vertex with list size at least three. The dashed box indicates
$X$ and a label on an edge is the common color of lists of its endpoints.}
\label{fig-defineX} 
\end{figure}

\section{Lists of size 3 are necessary}

In this section we give a proof of Theorem~\ref{thm:3necessary}. The
construction is analogous to the construction that bipartite graphs
are not 2-choosable.

\begin{proof}[Proof of Theorem~\ref{thm:3necessary}] Let $k$
be given. Let $G$ be a complete bipartite graph with part $X$ of
size $(k-1)^{2}$ and another part of size 2 formed by vertices $a$
and $b$. Let $L$ be a list assignment assigning to $a$ a list of
colors $\{a_{1},\ldots,a_{k-1}\}$ and to $b$ a list of colors $\{b_{1},\ldots,b_{k-1}\}$.
To the other vertices, $L$ assigns distinct lists of form $\{a_{i},b_{j}\}$
where $1\leq i,j\leq k-1$. There are $(k-1)^{2}$ such lists which
is exactly the size of $X$. Notice that $|L(u)\cup L(v)|=k$ for
every edge $uv$. See Figure~\ref{fig-3needed} for a sketch of $G$
and $L$.

\begin{figure}[htp]
\begin{centering}
\includegraphics{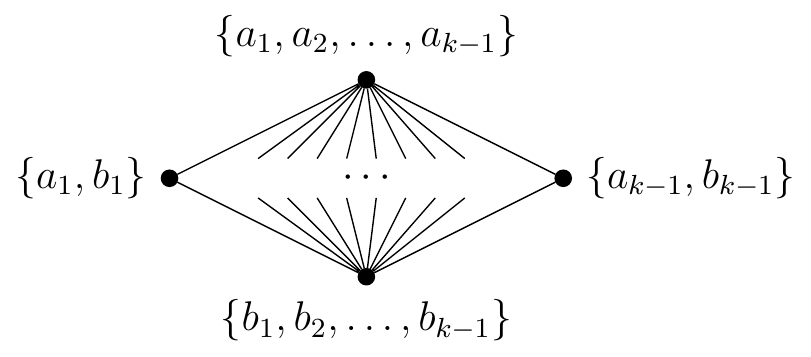} 
\par\end{centering}
\caption{Construction from Theorem~\ref{thm:3necessary}.}
\label{fig-3needed} 
\end{figure}

Suppose that there is an $L$-coloring of $G$. It assigns colors
$a_{i}$ to $a$ and $b_{j}$ to $b$ for some $1\leq i,j\leq k-1$.
However, there is a vertex with list $\{a_{i},b_{j}\}$, a contradiction.
Hence $G$ is not $L$-colorable. \end{proof}

\end{document}